\documentclass{amsart}
\usepackage{amssymb}

\usepackage{mathpazo,courier}
\usepackage[scaled]{helvet}

\usepackage[usenames,dvipsnames]{xcolor}
\usepackage[pagebackref,colorlinks,linkcolor=BrickRed,citecolor=OliveGreen,urlcolor=Blue,hypertexnames=true]{hyperref}

\usepackage{enumerate}

\newcommand\MM{\mathcal{M}}
\newcommand\MS{\mathcal{S}}

\newcommand{\ignore}[1]{}

\newcommand\al\alpha
\newcommand\la\lambda
\newcommand\de\delta
\newcommand\ep\varepsilon
\newcommand\ph\varphi
\newcommand\si\sigma
\newcommand\ta\tau
\newcommand\ps\psi
\newcommand\Ph\Phi
\newcommand\Ps\Psi

\theoremstyle{definition}
\newtheorem{thm}{Theorem}
\newtheorem{lemma}[thm]{Lemma}

\newtheorem{pro}[thm]{Proposition}

\newtheorem*{rep@theorem}{\rep@title}
\newcommand{\newreptheorem}[2]{%
\newenvironment{rep#1}[1]{%
 \def\rep@title{#2 \ref{##1}}%
 \begin{rep@theorem}}%
 {\end{rep@theorem}}}

\newreptheorem{pro}{Proposition}
\newreptheorem{thm}{Theorem}

\title{On the hardness of deciding the finite convergence of Lasserre hierarchies}

\author[Luis Felipe Vargas]{Luis Felipe Vargas}
\address{Istituto Dalle Molle di Studi sull'Intelligenza Artificiale (IDSIA), Switzerland}
\email{luis.vargas@idsia.ch}

\subjclass{ }

\date{\today}


\begin{document}
\begin{abstract}
\begin{itemize}
A polynomial optimization problem (POP) asks for minimizing a polynomial function given a finite set of polynomial constraints (equations and inequalities). This problem is well-known to be hard in general, as it encodes many hard combinatorial problems. The Lasserre hierarchy is a sequence of semidefinite relaxations for solving (POP).  Under the standard archimedean condition, this hierarchy is guaranteed to converge asymptotically to the optimal value of (POP) (Lasserre, 2001) and, moreover, finite convergence holds generically (Nie, 2012). In this paper, we aim to investigate whether there is an efficient algorithmic procedure to decide whether the Lasserre hierarchy of (POP) has finite convergence. We show that unless P=NP there cannot exist such an algorithmic procedure that runs in polynomial time. We show this already for the standard quadratic programs. Our approach relies on characterizing when finite convergence holds for the so-called Motzkin-Straus formulation (and some variations of it) for the stability number of a graph.
\end{itemize}
\end{abstract}

%\iffalse
\maketitle
%\newpage
%\tableofcontents
%\newpage
%\fi

%********************************************************************************
% Section: Right into the middle of things
%********************************************************************************
\section{Introduction}
A polynomial optimization problem asks for minimizing a polynomial function subject to a set of polynomial constraints (equations and inequalities). More precisely, given a polynomial $f$, and two sets of polynomials ${\bf{g}}=\{g_1, \dots, g_m\}$ and ${\bf{h}}=\{h_1, \dots, h_l\}$,  the task is to find (or approximate) the infimum of the following problem:
\begin{align}\label{POP}\tag{POP}
f^*=\inf_{x\in K} f(x) ,
\end{align}
where 
\begin{align}\label{K}
K=\Big\{x\in \mathbb{R}^n: g_i(x)\geq 0 \text{ for } i=1, \dots, m \text{ and } h_i(x)=0 \text{ for } i=1,\dots, l\Big\}.
\end{align}
Polynomial optimization captures many hard problems in many areas such as combinatorial optimization, portfolio and energy optimization (see, for example, \cite{las09} and references therein). Consequently, solving (\ref{POP}) is hard in general. A common approach for approximating (and solving) (\ref{POP}) is given by the so-called Lasserre sum-of-squares hierarchy, which is a sequence of semidefinite programs of increasing size as we now describe. 
\\

 We first introduce some notation. Throughout the article $\mathbb{N}=\{1, 2, \dots, \}$ and $\mathbb{R}$ will denote, respectively, the set of positive integers and real numbers, and for an integer $n\in \mathbb{N}$, we set $[n]=\{1,2,\dots, n\}$. We denote by $\mathbb{R}[x]:=\mathbb{R}[x_1, x_2, \dots, x_n]$ the ring of polynomials with real coefficients in the variables $x_1, \dots, x_n$. For any $d\in \mathbb{N}$, $\mathbb{R}[x]_d$ denotes the set of polynomials with degree at most $d$. A polynomial $p$ is a \textit{sum of squares} if we have $p=\sum_{i=1}^mq_i^2$ for some other polynomials $q_1, \dots, q_m$.  The cone of sums of squares is denoted by $\Sigma$. 
%$$\Sigma:=\Big\{\sum_{i=1}^mq_i^2: m\in \mathbb{N}, q_1, q_2, \dots, q_m\in \mathbb{R}[x]\Big\}$$
We let $\Sigma_r=\Sigma\cap \mathbb{R}[x]_r$.  For an integer $r\in \mathbb{N}$, the \textit{$r$-th truncated quadratic module} generared by the set of polynomials ${\bf{g}}=\{g_1, \dots, g_m\}$ is defined as 
$$\MM({\bf{g}})_r:=\Big\{\sum\limits_{i=0}^{n} \sigma_ig_i: \sigma_i\in \Sigma, \text{deg}(\sigma_ig_i)\leq r \text{ for } i=0,1,\dots, m, \text{ and } g_0=1\Big\}.$$
The union
$$\MM({\bf{g}}):= \bigcup_{r=0}^{\infty} \MM({\bf{g}})_r$$
is called the {\em quadratic module} generated by ${\bf{g}}$. Given a set of polynomials \linebreak ${\bf{h}}=\{h_1, \dots, h_l\}$, the {\em $r$-th truncated ideal} generated by ${\bf{h}}$ is defined as 
$$I({\bf{h}})_r:=\Big\{\sum_{i=1}^l q_ih_i: q_i\in \mathbb{R}[x], \text{deg}(q_ih_i)\leq r \text{ for } i=1,2,\dots, l\Big\}.$$
The ideal generated by ${\bf{h}}$ is the union 
$$I({\bf{h}})=\bigcup_{r=0}^\infty I({\bf{h}})_r.$$

%\textit{quadratic module} generated by ${\bf{g}}$ is defined as 
%$$\MM({\bf{g}}):=\Big\{\sum\limits_{i=0}^{n} \sigma_ig_i: \sigma_i\in \Sigma \text{ for } i=0,1,\dots, m, \text{ and } g_0=1\Big\}.$$
%The \textit{truncated quadratic module} at level $r$ (generared by ${\textbf{g}}$), denoted by $\MM({\bf{g}})_r$, is the set of elements in the quadratic module $\MM({\bf{g}})$ for which the degree of all summands is bounded by $r$:
%$$\MM({\bf{g}})_r:=\Big\{\sum\limits_{i=0}^{n} \sigma_ig_i: \sigma_i\in \Sigma, \text{deg}(\sigma_ig_i)\leq r \text{ for } i=0,1,\dots, m, \text{ and } g_0=1\Big\}.$$

%We denote by $I({\bf{h}})$ the \textit{ideal} generated by ${\bf{h}}$:
%$$I({\bf{h}}):=\{\sum_{i=1}^l q_ih_i \text{ for some polynomials } q_i \text{ for } i\in [l]\}.$$
%Observe that, if for a polynomial $f$ we have
%$$f\in \MM({\bf{g}}) + I({\bf{h}}),$$
%then $f$ is nonnegative on $K$. 
The {\em Lasserre sum-of-squares hierarchy} of problem (\ref{POP}) \cite{las01} is defined as the following sequence of relaxations: for $r\in \mathbb{N}$,
\begin{align}
f^{(r)}=\sup\{\lambda: f-\lambda \in M({\bf{g}})_r + I({\bf{h}})_r\}.
\end{align}
Clearly, for all $r\geq 1$, we have $f^{(r)}\leq f^{(r+1)}\leq f^*$. We say that the sets of polynomials ${\bf{g}}=\{g_1, \dots, g_m\}$ and ${\bf{h}}=\{h_1, \dots, h_l\}$ satisfy the \textit{archimedean condition} if we have 
\begin{align}\label{arch}
N-\sum_{i=1}^nx_i^2\in \MM({\bf{g}}) + I({\bf{h}}) \quad \text{ for some $N\in \mathbb{N}$.}
\end{align}
If the archimedean condition is satisfied, then the hierarchy $f^{(r)}$ converges asymptotically to $f^*$ as $r\to \infty$. This is a consequence of the celebrated \textit{Putinar's Positivstellensatz} \cite{put}. 

We say that the Lasserre hierarchy $f^{(r)}$ has {\em finite convergence} if 
$$f^{(r)}=f^* \text{ for some $r\in \mathbb{N}$.}$$ 
Finite convergence does not hold in general (see, e.g., Example 6.19 in \cite{lau09}). However, it has been shown that finite convergence holds under some mild conditions. We recall some known results from the literature. Laurent \cite{lau07} showed that if the set of common {\em complex} zeros of $h_1, \dots, h_l$ is finite  (i.e, if the set \linebreak $\{x\in \mathbb{C}^n: h_1(x)=\dots=h_l(x)=0\}$ is finite), then the Lasserre hierarchy $f^{(r)}$ has finite convergence. Later, Nie \cite{nie13} extended this result to the case when the set of common {\em real} zeros is finite. That is, if $\{x\in \mathbb{R}^n: h_1(x)=\dots=h_l(x)=0\}$ is finite, then the Lasserre hierarchy $f^{(r)}$ has finite convergence. In addition, Nie \cite{nie12} showed that under the archimedean condition the Lasserre hierarchy has finite convergence {\em generically} (here, {\em generically} means that the property holds in the entire space of input data except on a set of Lebesgue measure zero). It has also been observed by Henrion and Lasserre \cite{hl03, hl05} that finite convergence holds in many numerical experiments.

The above results raise the question of whether it is possible to fully identify the cases in which finite convergence holds (or fails). More specifically, (given $f$, ${\bf{g}}$, and ${\bf{h}}$) does there exist an algorithmic procedure for deciding if the Lasserre hierarchy of problem (\ref{POP}) has finite convergence. We show that such a procedure cannot run in polynomial time unless P=NP. This is the main result of this paper. 

 \begin{thm}\label{ref}
The problem of deciding whether the Lasserre hierarchy of a polynomial optimization problem has finite convergence is NP-hard. 
\end{thm}
We remark that our complexity result holds, in fact, under the archimedean condition, and more specifically for the class of standard quadratic programs, i.e., when $f$ is a quadratic function and $K=\Delta_n$ is the standard simplex (when \linebreak ${\bf{g}}=\{x_1, x_2, \dots, x_n\}$, and ${\bf{h}}=\{\sum_{i=1}^nx_i-1\}$).

Our approach is based on considering variations of the so-called Motzkin-Straus formulation, which is a quadratic optimization problem for computing (the inverse of) the stability number of a graph, as we briefly now describe. Our results also permit to fully characterize the graphs for which the Lasserre hierarchy of the (original) Motzkin-Straus formulation has finite convergence. This forms the second main result of this article (see Theorem \ref{main-2}).
\subsection{Strategy of the proof}
Given a graph $G=(V=[n],E)$, a subset of vertices $S\subseteq V$ is called {\em stable} in $G$ if $\{i,j\}\notin E$ for all distinct $i,j\in S$. The {\em stability number} of $G$, denoted by $\alpha(G)$, is the maximum cardinality of a stable set in $G$. Computing $\alpha(G)$ is a well-known NP-hard problem \cite{karp}. The starting point of our approach is to consider the following formulation of $\frac{1}{\alpha(G)}$ as minimization of a quadratic function over the standard simplex 
$$\Delta_n=\Big\{x\in \mathbb{R}^n: \sum_{i=1}^nx_i=1, x_i\geq 0 \text{ for } i\in [n]\Big\}.$$ 
This formulation is due to Motzkin and Straus \cite{ms}. Let $G=([n], E)$ be a graph. Then we have
\begin{align}\label{motzkin-straus}\tag{M-S}
    \frac{1}{\alpha(G)}=\min\Big\{x^T(A_G+I)x: x\in \Delta_n\Big\}.
\end{align}
Here, $I$ is the $n\times n$ identity matrix, and $A_G$ is the adjacency matrix of $G$ (i.e., $(A_G)_{i,j}=1$ if $\{i,j\}\in E$ and $(A_G)_{i,j}=0$ if $\{i,j\}\notin E$).

It has been observed that certain variations of problem (\ref{motzkin-straus}) also find $\frac{1}{\alpha(G)}$ (Section \ref{sec-preliminaries}, see also \cite{lv1, gibb}). In particular, we consider the following reformulation (\ref{ms-e}) of $\frac{1}{\alpha(G)}$.  Let $e\in E$ be an edge of $G$, and let $G\setminus e=(V, E\setminus \{e\})$ be the graph obtained by deleting the edge $e$ from the graph $G$. Then, we have
\begin{align}\label{ms-e}\tag{M-S-$e$}
 \frac{1}{\alpha(G)}=\min\Big\{x^T(A_G+I + A_{G\setminus e})x: x\in \Delta_n\Big\}.
\end{align}
In this paper, we study the finite convergence of the Lasserre hierarchy for problem (\ref{motzkin-straus}) and some variations of it, in particular the ones of the form (\ref{ms-e}). The notion of critical edges plays a crucial role in this analysis, as previously observed in \cite{lv1}. An edge $e\in E$ in a graph $G$ is {\em critical} if $\alpha(G\setminus e)=\alpha(G)+1$. For example, in the {\em odd cyles} $C_{2n+1}$ (for $n\geq 1$) all edges are critical, while no edge is critical in an {\em even cycle} $C_{2n}$ (for $n\geq 2$). A pair of vertices $(i,j)$ is called a {\em twin pair} if $\{i,j\}\in E$ and their set of neighbours are equal (excluding $i$ and $j$), i.e., $N_G(i)\setminus \{j\} = N_G(j)\setminus \{i\}$. If $(i,j)$ is a twin pair in $G$, then contracting the edge $\{i,j\}$ in $G$ produces the same graph as deleting the node $i$, denoted by $G\setminus i$ (and isomorphic to $G\setminus j$).

 Under the assumption that $G$ has no twin pairs we will characterize when the Lasserre hierarchy of problem (\ref{ms-e}) has finite convergence. 
\begin{pro}\label{main-e}
Assume $G$ has no twin pairs and let $e$ be an edge of $G$. The Lasserre hierarchy of problem (\ref{ms-e}) has finite convergence if and only if the edge $e$ is not critical in $G$.
\end{pro} 
The result of Proposition \ref{main-e} combined with the fact that deciding whether an edge is critical in a graph without twin pairs is NP-hard (see Theorem \ref{theo-find-critical}) gives the main result of our paper (Theorem \ref{ref}). We also characterize when the Lasserre hierarchy of (\ref{motzkin-straus}) has finite convergence.  This forms our second main result.
\begin{thm}\label{main-2}
Let $G$ be a graph. The Lasserre hierarchy of problem (\ref{motzkin-straus}) has finite convergence if and only if the graph obtained after contracting all twin pairs of $G$ does not have critical edges .
\end{thm}

We remark that the 'if part' in Proposition \ref{main-e} and  Theorem \ref{main-2} essentially follows from the previous work \cite{lv1}. This relies on using a result by Nie \cite{nie12} (see Theorem \ref{theo-nie}) claiming that finite convergence holds if the classical optimality conditions hold at every global minimizer. We describe these results in Section \ref{sec-preliminaries} for the sake of completeness. 

For the 'only if' part in Proposition \ref{main-e} and Theorem \ref{main-2}, we will exploit the structure of the (infnitely many) minimizers of problems (\ref{motzkin-straus}) and (\ref{ms-e}) to reach a contradiction. In fact, we show the result for a more general class of problems (see formulation (\ref{motzkin-straus-p}) in Section \ref{sec-preliminaries} and see Theorem \ref{theo-finite-convergence-w} for the more general statement).

\subsection{Link to related literature}
The complexity status of several decision problems in polynomial optimization has been studied recently. In \cite{az1}, Ahmadi and Zhang show that given two sets of polynomials ${\bf{g}}=\{g_1, \dots, g_m\}$ and \linebreak ${\bf{h}}=\{h_1, \dots, h_l\}$ it is NP-hard to test whether the archimedean property (\ref{arch}) holds. In the same work, they study the problem of deciding whether the optimum of problem (\ref{POP}) is attained. In particular, they show that this problem is NP-hard already when the objective is linear, and the constraints are quadratic. In addition, Ahmadi and Zhang \cite{az2} show that there is no polynomial-time algorithm which decides whether a quadratic program with a bounded feasible set has a unique local minimizer. In \cite{lv1}, we show that it is NP-hard to decide whether a polynomial optimization problem has finitely many minimizers, in fact we show this already for standard quadratic programs. The proof of the last two results also considers variations of the formulation (\ref{motzkin-straus}) by Motzkin and Straus, as we do in the present paper.
\\

The finite convergence of the Lasserre hierarchy of problem (\ref{motzkin-straus}) has been studied recently. Indeed, in \cite{lv1} we show that finite convergence holds for graphs without critical edges (also recalled later in Theorem \ref{finite-acritical}). In fact, we show that this result extends for several formulations for the stability number of {\em weighted} graphs. The following reformulation of problem (\ref{motzkin-straus}) has been also considered: 
$$\frac{1}{\alpha(G)}=\min \{(x^{\circ2})^T(A_G+I)x^{\circ2} : \sum_{i=1}^nx_i^2=1\},$$
where $x^{\circ2}=(x_1^2, x_2^2, \dots, x_n^2)$. Recently we have shown in \cite{sv} that the Lasserre hierarchy for this problem has {\em always} finite convergence, for any graph $G$.

 The Lasserre hierarchy of a (POP) has finite convergence if and only if certain polynomial (namely, $f-f^*$) which is nonnegative ({\em with zeros}) on the semialgebraic set $K$ belongs to the quadratic module $\MM({\bf{g}}) +I({\bf{h}})$. The representation of nonnegative polynomials in quadratic modules has been intesively studied in the literature. We highlight the works by Putinar \cite{put}, Scheiderer \cite{sch1,sch2}, Marshall \cite{mar03,mar06}, Kriel and Schweighofer\cite{ksa,ksb}, and references therein.
 
\subsection{Organization of the paper}
In Section \ref{sec-preliminaries} we recall the formulation (\ref{motzkin-straus}) by Motzkin and Straus and some variations of it. Then, we recap some known results about their minimizers and their Lasserre hierarchies.  In addition, we recall a result by Nie that permits to show finite convergence of the Lasserre hierarchy under the classical optimality conditions. In Section \ref{sec-characterization}, we prove Theorem \ref{main-2}, that is, we characterize the graphs for which the Lasserre hierarchy of problem (\ref{motzkin-straus}) has finite convergence. In fact, we show a more general result (Theorem \ref{theo-finite-convergence-w}) that will be used later for showing our main complexity result. Finally, in Section \ref{sec-complexity} we show our main result. Namely, we show that the problem of deciding whether the Lasserre hierarchy of a given polynomial optimization has finite convergen is NP-hard.
\\
\\
The main results of this article were included in the PhD thesis \cite{var}.

\section{Preliminaries}\label{sec-preliminaries}
\subsection{Stability number of a graph}
We start by recalling the result by Motzkin and Straus that gives the following formulation (\ref{motzkin-straus}) for $\frac{1}{\alpha(G)}$, already mentioned in the Introduction. 
\begin{thm}\label{thm-ms}
Let $G=([n], E)$ be a graph, then 
\begin{align}\label{motzkin-straus}\tag{M-S}
    \frac{1}{\alpha(G)}=\min\Big\{x^T(A_G+I)x: x\in \Delta_n\Big\}.
\end{align}
\end{thm}
Observe that given a graph $G$ and a stable set $S$ of $G$ of size $\alpha(G)$ the vector $x=\frac{1}{\alpha(G)}\chi^{S}$ is a minimizer of problem (\ref{motzkin-straus}). Here, $\chi^S\in \mathbb{R}^n$ is the characteristic vector of $S$ (defined by $(\chi^S)_i = 1$ for $i \in S$ and $(\chi^S)_i = 0$ for $i \in [n] \setminus S$). In general, problem (\ref{motzkin-straus}) might have infinitely many minimzers as shown in \cite{gibb} (see also \cite{lv1} and Theorem \ref{theorem-minimizers}).

In this paper, we consider some variations of problem (\ref{motzkin-straus}). For this, given a graph $G$ we consider the following set of matrices
\begin{align}\label{eqsetB}
\mathcal{B}(G)=\{B\in \MS^n: B_{ii}=1, B_{ij}\geq 1 \text{ for } \{i,j\}\in E, B_{ij}=0 \text{ for } \{i,j\}\notin E\Big\}. 
\end{align}
So, $A_G+I\in \mathcal{B}$ and $B\geq A_G +I$ for any $B\in \mathcal{B}$. Then,  for any $B\in \mathcal{B}$ the following equality holds as a consequence of Theorem \ref{thm-ms}:
\begin{align}\label{motzkin-straus-p}\tag{M-S-B}
\frac{1}{\alpha(G)}=\min\Big\{x^TBx: x\in \Delta_n\Big\}.
\end{align}
Indeed, the vectors of the form $x=\frac{1}{\alpha(G)}\chi^S$ (with $S$ stable in $G$ of size $\alpha(G)$) give $x^TBx=\frac{1}{\alpha(G)}$, and for any $x\in \Delta_n$ we have $x^TBx \geq x^T(A_G+I)x$.
 We now consider the Lasserre hierarchy corresponding to problem (\ref{motzkin-straus-p}).
 
 For convenience, we denote by $\MM(x_1, \dots, x_n)$ (resp. $\MM(x_1, \dots, x_n)_r$) the quadratic module (resp. the $r$-th truncated quadratic module) generated by the set of polynomials $\{x_1, x_2, \dots, x_n\}$. Also, we denote by $I_{\Delta_n}$ the ideal generated by the polynomial $\sum_{i=1}^nx_i-1$. The Lasserre relaxation of order $r$ for problem (\ref{motzkin-straus-p}) reads

\begin{align}\label{las-motzkin-B}
f_{G,B}^{(r)}=\max \{ \lambda: x^TBx-\lambda \in \MM(x_1, x_2, \dots, x_n)_r + I_{\Delta_n} \}.
\end{align}
When $B=A_G+I$, we simply write $f_G^{(r)}$. That is, we define $f_G^{(r)}:=f_{G,A_G+I}^{(r)}$, i.e., 
\begin{align}\label{las-motzkin}
f_{G}^{(r)}=\max \{ \lambda: x^T(A_G+I)x-\lambda \in \MM(x_1, x_2, \dots, x_n)_r + I_{\Delta_n} \}.
\end{align}

It is a well-know fact that the archimedean condition (\ref{arch}) is satisfied by the sets ${\bf{g}}=\{x_1, \dots, x_n\}$ and ${\bf{h}}=\{\sum_{i=1}^nx_i-1\}$. We repeat here the easy argument for the sake of completeness. We have the following identities
$$1-x_i= 1-\sum_{k=1}^n x_k +\sum_{k\in[n]\setminus \{i\}}x_k, \quad 1-x_i^2 = {(1+x_i)^2\over 2}(1-x_i) + {(1-x_i)^2\over 2}(1+x_i),$$
that imply $n-\sum_{i=1}^n x_i^2 \in \mathcal{M}(x_1, \dots, x_n) +I_{\Delta_n}$.

It is also well-known that the program (\ref{las-motzkin-B}) (and thus (\ref{las-motzkin})) is feasible and attains its maximum, thus the `sup' from the definition of the Lasserre hierarchy can be replaced by a 'max'. We repeat the argument for completeness as well. If $B\succeq 0$ (i.e., all its eigenvalues are nonnegative), then the polynomial $x^TBx$ is a sum of squares and thus $x^TBx\in \MM(x_1, \dots, x_n)$. Otherwise, the matrix $B-\mu I\succeq 0$, where $\mu:=\lambda_{\text{min}}(B)<0$. Then, the polynomial
$$x^TBx-n\mu=x^T(B-\mu I)x -\mu (n-\sum_{i=1}^nx_i^2)$$
belongs to $\MM(x_1, x_2, \dots, x_n) + I_{\Delta_n}$, showing feasibility of program (\ref{las-motzkin-B}). Finally, the optimum of program (\ref{las-motzkin-B}) is attained because the feasible region is a closed set (see \cite{mar03}).
\\

It has been shown in \cite{lv1} that the notion of critical edge plays a role in the analysis of the minimizers of problem (\ref{motzkin-straus-p}) and of the convergence of the hierarchies $f_{G,B}^{(r)}$ and $f_G^{(r)}$. Recall that an edge $e\in E$ is critical if $G$ if $\alpha(G\setminus e)=\alpha(G)+1$. For a verctor $x\in \mathbb{R}^n$, we denote by $\text{supp}(x)=\{i\in [n]: x_i>0\}$ its support. We have the following result.
\begin{thm}[particular case of Corollary 4.4. in \cite{lv1}]\label{theorem-minimizers}
Let $G=(V,E)$ be a graph, and let $B\in \mathcal{B}(G)$. Let $x\in \Delta_n$ with support $S=\text{supp}(x)$, and let $C_1, C_2, \dots, C_k$ denote the connected components of the graph $G[S]$. Then, $x$ is a minimizer of problem (\ref{motzkin-straus-p}) if and only if the following holds:
\begin{description}
\item[(i)]  $B_{ij}=1$ for all edges $\{i,j\}$ of $G[S]$,
\item[(ii)] $C_h$ is a clique of $G$ for all $h\in [k]$,
\item [(iii)] $\sum_{i\in C_h} x_i = \frac{1}{\alpha(G)}$ for all $h\in [k]$. 
\end{description}
In this case all edges of $G[S]$ are critical in $G$.
\end{thm}
This result is used to study the finite convergence of the hierarchies $f_{G,B}^{(r)}$, as described in the following sections.
%\begin{obs}
%An edge $\{i,j\}$ of a graph $G=(V,E)$ is critical if there exist two stable set $A,B\subseteq V$
%\end{obs}
\subsection{Optimality conditions and finite convergence}
In this section, we recall a result of Nie for proving the finite convergence of Lasserre hierarchies when the classical optimality conditions hold. We first do a quick recap of the classical optimality conditions applied to problem (\ref{POP}) (although it holds for a more general class of problems as described in \cite{ber}): 

Let $u$ be a local minimizer of problem (\ref{POP}), and let 
$$J(u)=\{j\in [m]: g_j(u)=0\}$$
be the index set of the active inequality constraints at $u$. We say that the {\em constraint qualification condition (CQC)} holds at $u$ if the vectors 
 $$\{\nabla g_j(u): j\in J(u)\}\cup \{\nabla h_i(u): i\in[l]\}$$
 are linearly independent. If (CQC) holds at $u$, then there exist multipliers $\lambda_1,\dots$, $\lambda_k$, $\mu_1, \dots, \mu_m\in \mathbb{R}$ such that
 \begin{align*}
 \nabla f(u) = \sum_{i=1}^l \lambda_i\nabla h_i(u) + \sum_{j=1}^m \mu_j\nabla g_j(u),\\
 \mu_1g_1(u)=0, \dots, \mu_mg_m(u)=0, \mu_1\geq 0, \mu_m\geq 0.
 \end{align*}
If it holds that 
$$\mu_j>0 \text{ for every } j\in J(u),\text{ and }\quad  \mu_j=0 \text{ for } j\in [m]\setminus J(u),$$
then we say that the {\em strict complementarity condition (SCC)} holds at $u$. We define now the Lagrangian function 
$$L(x)=f(x)-\sum_{i=1}^l \lambda_ih_i(x)-\sum_{j\in J(u)}\mu_jg_j(x).$$
A necessary condition for $u$ to be a local minimizer is the following: 
$$v^T\nabla^2L(u)v\geq 0 \text{ for all } v\in G(u)^\perp,$$
where $G(u)$ is defined as the matrix with rows the gradients of the active constraints at $u$, and $G(u)^\perp$ is its kernel. Hence, we have
\begin{align*}
G(u)^{\perp}=\{x\in \mathbb{R}^n: \ &  x^T\nabla  g_j(u)=0  \text{ for all  } j\in J(u) \text{ and }  \\  &x^T\nabla h_i(u)=0 \text{ for all } i\in [l]\}.
\end{align*}
If it holds that 
\begin{align*}
    v^T\nabla^2L(u)v> 0 \text{ for all } 0\neq v\in G(u)^{\perp},
\end{align*}
then we say that the \textit{second order sufficiency condition (SOSC)} holds at $u$.
\\

The relation between the optimality conditions just described and the finite convergence of the Lasserre hierarchy of (\ref{POP}) is given by the following result of Nie \cite{nie12}.
\begin{thm}[\cite{nie12}]\label{theo-nie}
Assume the archimedean condition is satisfied by the sets ${\bf{g}}=\{g_1, \dots, g_m\}$ and ${\bf{h}}=\{h_1, \dots, h_l\}$. If the conditions (CQC), (SCC) and (SOSC) hold at every global minimizer of (\ref{POP}), then the Lasserre hierarchy $f^{(r)}$ has finite convergence. 
\end{thm}
\subsection{Finite convergence for graphs with no critical edges}
As an application of Theorem \ref{theorem-minimizers} and Theorem \ref{theo-nie} we obtain the following result.
 This a particular case of a more general result obtained in my previous work \cite{lv1} with Monique Laurent. Specifically, it follows by considering a particular case of Proposition 4.5. and Theorem 5.1. in \cite{lv1}.
 \begin{thm}[\cite{lv1}]\label{finite-acritical}
Let $G=([n], E)$ be a graph, and let $B\in \mathcal{B}$. The following three assertions are equivalent:
\begin{description}
\item [(i)] $B_{lm}>1$ for any critical edge $\{l,m\}$ of $G$. 
\item[(ii)] Problem (\ref{motzkin-straus-p}) has finitely many minimizers.
\item[(iii)] All minimizers of problem (\ref{motzkin-straus-p}) satisfy the optimality conditions (CQC), (SCC), (SOSC).
\end{description}
In addition, any of these three assertions imply the following:
\begin{description}
\item [(iv)] The hierarchy $f_{G,B}^{(r)}$ has finite convergence.
\end{description}
\end{thm}
In the next section, we will study situations in which assertion (iv) is equivalent to assertions (i)-(ii)-(iii).

\section{Characterizing finite convergence}\label{sec-characterization}
 In this section, we aim to characterize when the hierarchy $f_{G,B}^{(r)}$ (and $f_G^{(r)}$) has finite convergence. As a main result, we characterize the graphs for which the hierarchy $f_{G}^{(r)}$ (of problem (\ref{motzkin-straus})) has finite convergence (Theorem \ref{main-2}). For this, we will show a more general result (Theorem \ref{theo-finite-convergence-w}) that we will use also in Section \ref{sec-complexity} for showing our main complexity result (Theorem \ref{ref}). 
\subsection*{In the absence of twin pairs}
The result of Theorem \ref{finite-acritical} show {\em sufficient} conditions for the hierarchy $f_{G,B}^{(r)}$ to have finite convergence.
We show that when the graph $G$ does not have twin pairs these {\em sufficient} conditions are in fact {\em necessary}. 
\begin{thm}\label{theo-finite-convergence-w}
Let $G=([n],E)$ be a graph without twin pairs and let $B\in \mathcal{B}(G)$. The Lasserre hierarchy $f_{G,B}^{(r)}$ has finite convergence to $\frac{1}{\alpha(G)}$ if and only if, for any critical edge $\{l,m\}$ of $G$, we have $B_{lm}>1$.
\end{thm}

\begin{proof}
The `if' part follows from Theorem \ref{finite-acritical}.
For the `only if' part we proceed by contradiction as follows. Assume there is a critical edge $\{l,m\}$ such that $B_{lm}=1$. We assume, moreover, that $f_{G,B}^{(r)}$ has finite convergence, that is, there exist $\sigma$, $\sigma_i\in \Sigma$ for $i\in [n]$, and $q\in \mathbb{R}[{\bf{x}}]$ such that 
\begin{align}\label{aux-lasserre}
x^T Bx-\frac{1}{\alpha(G)}=\sigma + \sum_{i=1}^nx_i\sigma_i + q\Big(\sum_{i=1}^nx_i-1\Big).
\end{align}
Since the edge $\{l,m\}$ is critical, we observe that there exists $S\subseteq V$ such that $S\cup\{l\}$ and $S\cup\{m\}$ are stable of size $\alpha(G)$ in $G$. By Theorem \ref{theorem-minimizers}, for $t\in (0,1)$, the vector 
$$u_t=\frac{1}{\alpha(G)}(t\chi^{S\cup\{l\}} + (1-t)\chi^{S\cup\{m\}})$$
 is an optimal solution of problem (\ref{motzkin-straus-p}), i.e., $u_t^T Bu_t=\frac{1}{\alpha(G)}$ and $u_t\in \Delta_n$.  We evaluate relation (\ref{aux-lasserre}) at $x+u_t$ and we obtain 
\begin{align}\label{contra}
x^T Bx + 2x^T Bu_t = \sigma(x+u_t) + \sum_{i=1}^n(x+u_t)_i\sigma_i(x+u_t) + q(x+u_t)(\sum_{i=1}^nx_i).
\end{align}
Now, we will reach a contradiction by comparing coefficients at both sides in relation (\ref{contra}). First, since there is no constant term on the left hand side the constant term on the right hand side is equal to zero. That is, 
$$\sigma(u_t) + \frac{1}{\alpha(G)}\sum_{s\in S}\sigma_s(u_t) + \frac{t}{\alpha(G)}\sigma_l(u_t) + \frac{1-t}{\alpha(G)}\sigma_m(u_t)=0.$$
This implies that, for any $t\in (0,1)$, the polynomials $\sigma(x+u_t)$, and $\sigma_i(x+u_t)$ (for $i\in S\cup\{l,m\}\}$) do not have a constant term and therefore do not have linear terms. Now, we compare the coefficient of $x_s$, where $s\in S$. On the left hand side of (\ref{contra}), it is equal to 
$$2\sum_{i\in S\cup\{l,m\}}B_{si}(u_t)_i= 2B_{ss}(u_t)_s=\frac{2}{\alpha(G)}.$$
 On the right hand side of (\ref{contra}), the polynomials $\sigma(x+u_t)$ and $(x+u_t)_i\sigma_i(x+u_t)$ for $i\in S\cup\{l,m\}$ have no linear term, and for $i\in [n]\setminus (S\cup\{l,m\})$, the polynomials $(x+u_t)_i\sigma(x+u_t)$ are divisible by $x_i$. Hence, the coefficient of $x_s$ is $q(u_t)$. Therefore, $q(u_t)=\frac{2}{\alpha(G)}$. Let $j\in [n]$ be such that $j\in N_G(l)$ and $j\notin N_G(m)$. Here, we use that $l$ and $m$ are not twin nodes (we switch $l$ and $m$ if necessary). We compare the coefficient of $x_j$ at both sides of (\ref{contra}). In the left hand side, the coefficient of $x_j$ is $2\sum_{i\in S\cup\{l,m\}}B_{ij}(u_t)_i = \frac{2}{\alpha(G)}B_{lj}t + \frac{2}{\alpha(G)}\sum_{i\in S}B_{ij}$. Finally, On the right hand side, the coefficient of $x_j$ is $\sigma_j(u_t)+q(u_t)=\sigma_j(u_t)+\frac{2}{\alpha(G)}$. Hence, we obtain
$$ \sigma_j(u_t)= \frac{2}{\alpha(G)}B_{lj}t + \frac{2}{\alpha(G)}\sum_{i\in S}B_{ij} -\frac{2}{\alpha(G)}.$$
This is a contradiction because $\sigma_j(u_t)$ is a sum of squares of polynomials in $t$, while the polynomial in the right hand has degree 1 (since $B_{lj}\geq 1$).
\end{proof}

\subsection*{Finite convergence for the Motzkin-Straus formulation (\ref{motzkin-straus})}
Now we show Theorem \ref{main-2}. Namely, we characterize the graphs $G$ for which the Lasserre hierrarchy $f_G^{(r)}$ (of problem (\ref{motzkin-straus})) has finite convergence. For this, we first show that the finite convergence of the hierarchy $f_G^{(r)}$ is not affected by contracting twin pairs. 
\begin{lemma}\label{lemma-twin}
Let $G=\{[n], E\}$ be a graph such that $(1,2)$ is a twin pair. The hierarchy $f_G^{(r)}$ has finite convergence if and only if the hierarchy $f_{G\setminus 1}^{(r)}$ has finite convergence. 
\end{lemma}
\begin{proof}
Clearly, we have that $\alpha(G\setminus 1)=\alpha(G)$. We now show that $f_G^{(r)}=\frac{1}{\alpha(G)}$ if and only if $f_{G\setminus 1}^{(r)}=\frac{1}{\alpha(G)}$. 

The `only if' part follows directly by replacing $x_1=0$ in the definition of $f_G^{(r)}$ (see relation (\ref{las-motzkin})). Assume now that $f_{G\setminus1}^{(r)}=\frac{1}{\alpha(G)}$. Then, we have
\begin{align}\label{eq-twin}
x^T(A_{G\setminus 1}+I)x-\frac{1}{\alpha(G)}=\sigma + \sum_{i=2}^{n}x_i\sigma_i + q(\sum_{i=2}^nx_i-1),
\end{align}
for some $\sigma_0, \sigma_2, \dots, \sigma_n\in \Sigma_r$, $q\in \mathbb{R}[x]$.
Now, we set 
$$f_G=x^T(A_G+I)x.$$
Observe that  the following equality holds:
$$f_G(x_1, x_2, \dots)=f_{G\setminus 1}(x_1+x_2, x_3, \dots, x_n).$$
Then, by replacing $x_1$ by $x_1+x_2$ in relation (\ref{eq-twin}) we obtain 
$$x^T(A_{G}+I)x-\frac{1}{\alpha(G)}=\tilde{\sigma} + (x_1+x_2)\tilde{\sigma_2}+\sum_{i=3}^{n}x_i\tilde{\sigma_i} + \tilde{q}(\sum_{i=1}^nx_i-1),$$
for some $\tilde{\sigma}, \tilde{\sigma_2},\dots, \tilde{\sigma_n}\in \Sigma_r$, $\tilde{q}\in \mathbb{R}[x]$. This shows $p_G^{(r)}=\frac{1}{\alpha(G)}$.
\end{proof}

Then, Theorem \ref{main-2} follows directly from Theorem \ref{theo-finite-convergence-w} and Lemma~\ref{lemma-twin}. We repeat the statement for the sake of clarity. 
\begin{repthm}{main-2}
Let $G$ be a graph. The Lasserre hierarchy of problem (\ref{motzkin-straus}) has finite convergence if and only if the graph obtained after contracting all twin pairs of $G$ does not have critical edges .
\end{repthm}

 \section{Complexity result}\label{sec-complexity}
 In this section, we show the main result of the paper (Theorem \ref{ref}). Namely, we resolve the complexity status of the problem \textbf{FINITE-CONV}, defined as follows.
\\
\\
\textbf{FINITE-CONV:} Given a polynomial optimization problem (POP), does the corresponding Lasserre hierarchy $f^{(r)}$ has finite convergence to $f^*$?
\\

We will show that (FINITE-CONV) is NP-hard, already for standard quadratic programs. For this, we will consider the class of problems (\ref{motzkin-straus-p}), where \linebreak $B\in \mathcal{B}(G)$ and $G$ is a graph. As mentioned in the introduction, the formulation by Motzkin and Straus for $\frac{1}{\alpha(G)}$ has already been used for resolving the complexity status of several problems in continuous optimization.  

\subsection{Linear programs}
Consider first the case when (POP) is a linear optimization problem of the form
\begin{align}\label{linear-opt}\tag{L-P}
p^*=\inf \{ c^Tx: Ax\geq b, Dx=f\},
\end{align}
where $A\in \mathbb{R}^{n\times m}$, $D\in \mathbb{R}^{n\times l}$, $b\in \mathbb{R}^m$, $f\in \mathbb{R}^l$, $c\in \mathbb{R}^n$. Let $a_i$ (for $i\in [m])$ and $d_i$ (for $i\in [l])$ denote, respectively,  the rows of the matrices $A$ and $D$. Then, problem (\ref{linear-opt}) reads

\begin{align}\tag{L-P}
\begin{split}
p^*=\inf \Big \{c^T x : & a_i^T x - b_i\geq 0 \text{ for } i\in[m]
\\ & d_i^Tx- f_i=0 \text{ for } i\in [l] \Big \}.
\end{split}
\end{align}
We show that for the class of problems (L-P) (FINITE-CONV) can be solved in polynomial time. Observe that the first level of the Lasserre sum-of-squares hierarchy for problem (\ref{linear-opt}) finds its optimum $p^*$.
Indeed, the first level of the hierarchy $p^{(1)}$ reads
\begin{align*}
p^{(1)}=\sup_{\substack{\lambda \in \mathbb{R}, \alpha\in \mathbb{R}_+^m \\ \beta\in \mathbb{R}^l}} \{\lambda : c^T x -\lambda =  \sum_{i=1}^m\alpha_i(a_i^T x -b_i) + & \sum_{i=1}^l \beta_i(d_i^T-f_i)\}
\end{align*}
By equating coefficients at both sides we obtain that this problem reads 
\begin{align*}
p^{(1)}=\sup_{\alpha\in \mathbb{R}^n_+, \beta\in \mathbb{R}^l} \{\alpha^Tb + \beta^Tf: \quad c= A^T\alpha + D^T\beta \}.
\end{align*}

Note that this is precisely the dual linear program of (\ref{linear-opt}). Hence, finite convergence always holds for linear programs, already at level $r=1$.

\subsection{Hardness in standard quadratic programs}

We now show that the problem (FINITE-CONV) is NP-hard already for the class of standard quadratic programs. Our approach consists of using the results from Section \ref{sec-characterization}, combined with the fact that deciding whether an edge is critical in a graph without twin pairs is an NP-hard problem. We consider the following two problems. 
\vspace{0.1cm}
\\
\textbf{CRITICAL-EDGE*:} Given a graph $G=(V,E)$ without twin pairs and an edge $e\in E$, is $e$ a critical edge of $G$?
\\
\vspace{0.1cm}
\\
\textbf{STABLE-SET:} Given a graph $G$ and $k\in \mathbb{N}$, does $\alpha(G)\geq k$ hold?
\\

The problem STABLE-SET is well-known to be NP-Complete \cite{karp}. From this, we now prove that unless P=NP there is no polynomial-time algorithm for solving CRITICAL-EDGE*. \footnote{In a previous work with Monique Laurent \cite{lv1} we show that the problem of deciding whether an edge is critical in a graph is NP-hard. Here, we extend this argument to the problem CRITICAL-EDGE*, which has the additional condition that $G$ has no twin pairs. }

\begin{thm}\label{theo-find-critical}
If there is a polynomial-time algorithm that solves the problem \linebreak {\rm CRITICAL}-{EDGE*}, then P=NP.
\end{thm}
\begin{proof}
Assume there is a poly-time algorithm A for solving CRITICAL-EDGE*. We show that we can find the stability number of an arbitrary graph in polynomial time. Let $G$ be a graph. We can check whether $G$ has twin pairs in polynomial time by checking each pair of nodes and their set of neighbors. If there is a twin pair $(u,v)$, then update the graph $G\to G\setminus u$ by deleting the node $u$, and we have $\alpha(G)=\alpha(G\setminus u)$. We repeat the procedure until $G$ has no twin pairs. Then, we take an edge $e\in E$ and, using the algorithm A, we check if $e$ is critical in $G$. We update graph $G\to G\setminus e$ by deleting the edge $e$, for which $\alpha(G\setminus e)=\alpha(G)$ if $e$ is not critical and $\alpha(G\setminus e)=\alpha(G)+1$ if $e$ is critical. We stop if the graph has no edges. This process is going to finish since at any step we delete a node or an edge. The algorithm will finish with a graph $\tilde{G}$ with $|V|-d$ nodes, where $d$ is the number of node deletions done in the process. Then, we have $\alpha(\tilde{G})= |V|-d=\alpha(G)+ c$, where $c$ is the number of times we found a critical edge at the edge deletion step. Hence, we can compute $\alpha(G)$ in poly-time using algorithm A.
\end{proof}
We will now use this complexity result to settle the complexity status of problem FINITE-CONV. For this, let $G=(V, E)$ be a graph without twin pairs and let $e\in E$. Recall problem (\ref{ms-e}) we considered already in the Introduction:
\begin{align}\label{motzkin-e}\tag{M-S-$e$}
\frac{1}{\alpha(G)}= \min \Big\{ x^T(A_G + I +A_{G\setminus e})x : x\in \Delta_n\Big\}.
\end{align}
This is a problem of the form (\ref{motzkin-straus-p}) with $B=A_G+I+A_{G\setminus e}$. As direct application of Theorem \ref{theo-finite-convergence-w} we obtain Proposition \ref{main-e}. We repeat the statement for clarity.

\begin{reppro}{main-e}
Assume $G$ has no twin pairs and let $e$ be an edge of $G$. The Lasserre hierarchy of problem (\ref{ms-e}) has finite convergence if and only if the edge $e$ is not critical in $G$.
\end{reppro}

Combining Theorem \ref{theo-find-critical} and Proposition \ref{main-e}, we obtain the result of Theorem \ref{ref}, which forms the main contribution of this paper. We repeat the statement for the sake of clarity. 
 \begin{repthm}{ref}
The problem of deciding whether the Lasserre hierarchy of a polynomial optimization problem has finite convergence is NP-hard. 
\end{repthm}

\section*{Acknowledgements}

I thank Monique Laurent for the useful discussions and her valuable comments about the presentation of this work.  This work is supported by the European Union's Framework Programme for Research and Innovation Horizon
2020 under the Marie Skłodowska-Curie Actions Grant Agreement No. 813211  (POEMA) and by the Swiss National Science Foundation project n. 200021$\_$207429 / 1 ``Ideal Membership Problems and the Bit Complexity of Sum of Squares Proofs''. Most of this reasearch was carried out when the author was with Centrum Wiskunde \& Informatica, Amsterdam.

\end{document}